\documentclass[12pt,reqno]{amsart}
\usepackage[latin1]{}
\usepackage{graphicx}
\usepackage{tikz-cd}
\usepackage{graphicx}
\usepackage{caption}
\usepackage{tabularx}
\usepackage{longtable}
\usepackage{tabularray}
\usepackage{amsmath, amssymb}
\usepackage{array}
\usepackage[english]{babel}  % catalan
\usepackage{amsmath}
\usepackage{longtable}
\usepackage{mathtools}
\usepackage{multicol}
\usetikzlibrary{positioning}
\usepackage{rotating}
\usepackage{verbatim} 

\usepackage{graphicx}
\usepackage{tikz-cd}
\usepackage{caption}

\usepackage{amssymb,tikz}
\usetikzlibrary{positioning}
\usepackage[colorlinks=false,urlbordercolor=white]{hyperref}

\def\gn#1#2{{$\href{http://groupnames.org/\#?#1}{#2}$}}
\def\gn#1#2{$#2$}  % comment this line out to get html links
\tikzset{sgplattice/.style={inner sep=1pt,norm/.style={red!50!blue},char/.style={blue!50!black},
  lin/.style={black!50}},cnj/.style={black!50,yshift=-2.5pt,left=-1pt of #1,scale=0.5,fill=white}}

\newtheorem{theorem}{Theorem}[section]
\newtheorem{definition}[theorem]{Definition}
\newtheorem{conjecture}[theorem]{Conjecture}
\newtheorem{proposition}[theorem]{Proposition}

\newtheorem{lemma}[theorem]{Lemma}

\newtheorem{remark}[theorem]{Remark}

\newcommand{\End}{\operatorname{End}}
\newcommand{\Q}{\mathbb Q}
\newcommand{\Qbar}{{\overline{\mathbb Q}}}

\newcommand{\Z}{\mathbb Z}
\newcommand{\F}{\mathbb F}

\newcommand{\C}{\mathbb C}
\newcommand{\GL}{\mathrm{GL}}

\newcommand{\Frob}{\operatorname{Frob}}

\newcommand{\Ver}{\operatorname{Ver}}
\newcommand{\cB}{{\mathcal B}}

\newcommand{\Gal}{\operatorname{Gal}}

\newcommand{\cO}{\mathcal{O}}
\newcommand{\cD}{\mathcal{D}}
\newcommand{\cL}{\mathcal{L}}

\newcommand{\gm}{\mathfrak{m}}
\newcommand{\gp}{\mathfrak{p}}

\newcommand{\gl}{\mathfrak{l}}

\newcommand{\gf}{\mathfrak{f}}

\newcommand{\sss}{\operatorname{ss}}

\theoremstyle{definition}

\newenvironment{sketchproof}
{\begin{proof}[Sketch of proof]}
{\end{proof}}

\begin{document}
    
\title{
%(Congruences between) \\ modular abelian surfaces 
%with inner-twist \\ 
%and CM modular forms}
Modular forms of CM type mod $\ell$}

\author[Dieulefait]{Luis Dieulefait}
\address[Dieulefait]{Universitat de Barcelona}

\author[González]{Josep González}
\address[González]{Universitat Politècnica de Catalunya}

\author[Lario]{Joan-C. Lario}
\address[Lario]{Universitat Politècnica de Catalunya}

%\affil[2]{Address of second author}

\date{\today. The first and third authors are funded by the project PID2022-136944NB-IOO (MECD)}

%\keywords{}

\begin{abstract}
We say that a normalized modular form is of CM type modulo $\ell$ by an imaginary quadratic field $K$ if its Fourier coefficients $a_p$ are congruent to $0$ modulo a prime $\mathcal{L}\mid \ell$ for every prime $p$ that is inert in $K$.

In this paper, we address the following question. Let $f$ be a weight~$2$ cuspidal Hecke eigenform without complex multiplication which is of CM type modulo $\ell$ by an imaginary quadratic field $K$. Does there exist a congruence modulo $\ell$ between $f$ and a genuine CM modular form of weight~$2$?

We conjecture that such a congruence always exists.
We prove this conjecture for $\ell>2$ and $\ell\neq 3$ when $K=\mathbb{Q}(\sqrt{-3})$. In this setting, we discuss three situations: (i) modular forms attached to abelian surfaces with quaternionic multiplication, (ii) $\mathbb{Q}$-curves completely defined over an imaginary quadratic field, and (iii) elliptic curves over $\mathbb{Q}$ whose $5$-torsion Galois representation has image the maximal cyclic of order $16$ inside $\GL_2(\F_5)$.

In all these cases, the modular forms under consideration are of CM type modulo suitable primes~$\ell$, and we show that the associated residual Galois representations are monomial with respect to an imaginary quadratic field $K$ (in some instances, more than one such field).

Finally, we present numerical evidence that motivated the conjecture and provides further support for its validity beyond the cases treated in this paper.
\end{abstract}

\flushbottom
\maketitle
\tableofcontents
\thispagestyle{empty}

\section{Introduction}

Let $f \in S_2(\Gamma_0(N))^{\mathrm{new}}$ be a normalized newform of weight $2$ and level $N$. We denote by
\[
f(q)=\sum_{n=1}^\infty a_n q^n
\]
its $q$-expansion, and by $E=\mathbb{Q}(\{a_n\})$ its Fourier coefficient field. Let $A/\mathbb{Q}$ be the abelian variety quotient of the Jacobian $J_0(N)$ attached to $f$ via Shimura’s construction. Then $\mathrm{End}_{\mathbb{Q}}(A)$ is isomorphic to an order in the ring of integers of $E$. Since we assume that $f$ has trivial Nebentypus, the field $E$ is totally real.

Let $\ell$ be a prime. For every prime ideal $\mathfrak{L}\mid \ell$ of $E$, there exists a Galois representation
\[
\varrho \colon \mathrm{Gal}(\overline{\mathbb{Q}}/\mathbb{Q}) \longrightarrow \mathrm{GL}_2(\mathbb{F}),
\]
where $\mathbb{F}$ denotes the residue field of $E$ at $\mathfrak{L}$. This representation is unramified outside $\ell N$ and satisfies
\[
\mathrm{tr}(\varrho(\mathrm{Frob}_p)) \equiv a_p \pmod{\mathfrak{L}}, 
\qquad 
\det(\varrho(\mathrm{Frob}_p)) = p
\]
for all primes $p \nmid \ell N$. The representation $\varrho$ describes the Galois action on the $\mathfrak{L}$-torsion module
\[
A[\mathfrak{L}] = \bigcap_{\alpha \in \mathfrak{L} \cap \mathrm{End}(A)} \ker(\alpha \colon A \to A).
\]
It is odd, and is surjective for all sufficiently large $\ell$ provided that $f$ does not have complex multiplication.

Recall that $f$ has complex multiplication (CM) by an imaginary quadratic field $K$ if there exists a nontrivial Dirichlet character $\varepsilon$ associated with $K$ such that
\[
a_p = \varepsilon(p)\, a_p
\]
for all primes $p \nmid N$. Equivalently, $a_p=0$ for all primes $p \nmid N$ that are inert in $K$.

The definition and question we introduce below arise naturally in the context of congruences between modular forms and constitute the central theme of this article.

\begin{definition}
Let $f(q)=\sum_{n=1}^\infty a_n q^n \in S_2(\Gamma_0(N))^{\mathrm{new}}$ be a normalized newform and let $\ell$ be a prime. We say that $f$ is of CM type modulo $\ell$ if there exist a prime ideal $\mathfrak{L}\mid \ell$ of $E$ and an imaginary quadratic field $K$ such that
\[
a_p \equiv 0 \pmod{\mathfrak{L}}
\]
for all primes $p \nmid \ell N$ that are inert in $K$. In this case, we also say that $f$ is of CM type modulo $\ell$ by~$K$.
\end{definition}

It is easy to construct examples of newforms that are of CM type modulo $\ell$ but do not have genuine CM. The examples presented in Section~\ref{examples} led us to formulate the following conjecture.

\begin{conjecture}\label{conj}
Let $f(q)=\sum_{n=1}^\infty a_n q^n \in S_2(\Gamma_0(N))^{\mathrm{new}}$ be a normalized newform of CM type modulo $\ell$ by $K$. Then there exists a normalized newform
\[
g(q)=\sum_{n=1}^\infty b_n q^n \in S_2(\Gamma_0(M))^{\mathrm{new}}
\]
with CM by $K$ such that
\[
a_p \equiv b_p \pmod{\mathfrak{L}}
\]
for all primes $p \nmid NM\ell$, where $\mathfrak{L}\mid \ell$ is a prime ideal in the compositum of the Fourier coefficient fields of $f$ and $g$.
\end{conjecture}

Variants of this conjecture have been studied in \cite{Billerey}, when the Serre weight of the residual representation lies in $[2,\ell-1]$, and in \cite{Nua} for residual dihedral representations of Serre weight~$2$. Our main result is the following.

\begin{theorem}\label{thm:main}
Assume that $\ell>2$, and, in the case $K=\mathbb{Q}(\sqrt{-3})$, that $\ell \neq 3$. Then Conjecture~\ref{conj} holds.
\end{theorem}

The structure of the paper is as follows. In Section~2, we review the basic properties of residual Galois representations attached to CM newforms with trivial Nebentypus, including their monomial nature, and recall a result of Billerey--Mortarino relevant to our setting. In Section~3, we prove a lifting result (Theorem~\ref{deformation}) that constructs Hecke characters from mod $\ell$ characters under suitable hypotheses, and deduce Theorem~\ref{thm:main}.

Sections~4--7 are devoted to three families of modular forms of CM type modulo $\ell$. Two of these arise from abelian surfaces with additional endomorphisms defined over an imaginary quadratic field (either with quaternionic multiplication or the square of a $\mathbb{Q}$-curves), while the third corresponds to a family of elliptic curves over $\mathbb{Q}$ of CM type modulo $5$. In all cases, we describe explicitly the associated Galois representations as induced from characters of $\mathrm{Gal}(\overline{\mathbb{Q}}/K)$.

Finally, in Section~\ref{examples} we present numerical evidence that motivated the conjecture and supports the validity of our results.

%---------------
\section{Galois representations arising from CM modular forms}

\noindent Keeping the notation of the Introduction, the following lemma exhibits the connection between CM newforms and monomial representations (i.e., those induced from characters).

\begin{lemma}\label{Nualart}
Assume that $f \in S_2(\Gamma_0(N))^{\mathrm{new}}$ is a newform with CM by an imaginary quadratic field~$K$. Let $\varrho$ be the Galois representation attached to $f$ and a prime $\mathfrak{L}\mid \ell$ as above, and let $n$ be the residue degree of $\mathfrak{L}$. Then:
\begin{itemize}
\item[(i)] There exists a character 
\[
r \colon \mathrm{Gal}(\overline{\mathbb{Q}}/K) \longrightarrow \mathbb{F}_{\ell^{2n}}^\times
\]
such that $\mathrm{Ind}_K^{\mathbb{Q}}(r) \simeq \varrho$.

\item[(ii)] If $\ell>2$, then the associated projective representation
\[
\overline{\varrho} \colon \mathrm{Gal}(\overline{\mathbb{Q}}/\mathbb{Q}) \longrightarrow \mathrm{PGL}_2(\mathbb{F})
\]
has image isomorphic to either a dihedral group $D_n$ of order $2n\geq 6$ or the cyclic group $C_2$. The dihedral case occurs if and only if $\mathfrak{L}$ is unramified in $EK/E$. In the cyclic case, $\overline{\varrho}$ is the quadratic character attached to $K$.
\end{itemize}
\end{lemma}

\begin{proof}
By Shimura~\cite{Shi}, there exists an integral ideal $\mathfrak{m}$ and a Hecke character 
\[
\psi \colon I(\mathfrak{m}) \longrightarrow \mathbb{C}^\times,
\]
where $I(\mathfrak{m})$ denotes the group of fractional ideals of $K$ coprime to $\mathfrak{m}$, such that
\[
f(q) = \sum_{\mathfrak{a}} \psi(\mathfrak{a})\, q^{\mathrm{N}(\mathfrak{a})}
= \sum_{n=1}^\infty a_n q^n,
\]
where $\mathfrak{a}$ runs over integral ideals coprime to $\mathfrak{m}$. The number field $\mathbb{Q}(\{\psi(\mathfrak{a})\})$ generated by the values of $\psi$ is the quadratic extension $EK$ of the Fourier coefficient field $E$ of $f$.

Since $f$ has trivial Nebentypus, one has $\mathfrak{m} = \overline{\mathfrak{m}}$ (see Lemma~2.1 of~\cite{GL2}). In particular, if $\mathfrak{a} \in I(\mathfrak{m})$, then also $\overline{\mathfrak{a}} \in I(\mathfrak{m})$.

Let $\mathfrak{l}$ be a prime of $EK$ above $\ell$, and define a character
\[
r \colon \mathrm{Gal}(\overline{\mathbb{Q}}/K) \longrightarrow \mathbb{F}_{\ell^{2n}}^\times,
\qquad
\mathrm{Frob}_{\mathfrak{p}} \longmapsto \psi(\mathfrak{p}) \bmod \mathfrak{l}.
\]
The induced representation 
$\operatorname{Ind}_K^\Q(r)$ satisfies:
$$
\operatorname{Ind}_K^\Q(r)(\Frob_p) = 
\begin{cases}
\begin{pmatrix}
    r(\Frob_{\mathfrak{p}}) & 0 \\
    0 &r(\Frob_{\overline{\mathfrak{p}}})
\end{pmatrix} 
& \text{if $(p)=\mathfrak{p}\overline{\mathfrak{p}}$ splits in $K$;}\\[18pt]
\begin{pmatrix}
    0 & r(\Frob_{p}c) \\
r(c\Frob_{p}) & 0
\end{pmatrix} 
& \text{if $(p)=\mathfrak{p}$ is inert in $K$,}\\
\end{cases}
$$
where $c$ stands for a complex conjugation in $\Gal(\Qbar/\Q)$. If $p$ is inert in $K$, then
$c \Frob_p=\Frob_p c= \sigma$ with $\sigma \in \Gal(\Qbar/K)$ and
$\Frob_p^2 = \sigma^2=\Frob_{\mathfrak{p}}$. Hence, one has 
$$
\begin{array}{l@{\,}l}
\operatorname{tr}
\operatorname{Ind}_K^\Q(r)(\Frob_p) &
\equiv a_p \bmod{\mathfrak{L}} \\[4pt]
\operatorname{det}
\operatorname{Ind}_K^\Q(r)(\Frob_p) & = p \bmod{\mathfrak{L}}.
\end{array}
$$
By the Chebotarev density theorem, it follows that $\mathrm{Ind}_K^{\mathbb{Q}}(r) \simeq \varrho$.

For the second statement, Nualart~\cite[Proposition~1]{Nua} shows that the image of the projective representation is dihedral if and only if $r \neq \overline{r}$, where $\overline{r}(\mathfrak{a}) = r(\overline{\mathfrak{a}})$.  Notice that $\overline{r}$ is denoted by ${\,}^\sigma r$ in \cite{Nua}.

If $r=\overline{r}$, then the projective representation of $\varrho$ coincides with the quadratic character attached to $K$ since $\varrho$ projects to $\operatorname{Id}$ on the Frobenius at split primes in $K$ and to 
the anti-identity matrix
on Frobenius at inert primes in $K$. In such a case, we have 
$\psi(\mathfrak{a}) \equiv  
\psi(\overline{\mathfrak{a}})\bmod \mathfrak{l}
$ for all $\mathfrak{a}$ in~$I(\mathfrak{m})$. Hence, 
$2\psi(\mathfrak{a})$ mod $\mathfrak{l}$ belongs to $\F_{\ell^n}$. Since we assume that $\ell>2$, then $\psi(\mathfrak{a})$ mod $\mathfrak{l}$ lands in~$\F_{\ell^n}$ for all $\mathfrak{a} \in I(\mathfrak{m})$.
This can occur only when 
$\mathfrak{L}$ ramifies in $EK$ since $[EK\colon E]=2$ as $E$ is totally real.
\end{proof}

\begin{remark}
If $f \in S_2(\Gamma_1(N))^{\mathrm{new}}$ has nontrivial Nebentypus, then the residual representation attached to $f$ is still induced from a character $r$ as above.
\end{remark}

\begin{remark}
Let $\mathcal{D}$ be the different of $K/\mathbb{Q}$ and let $\varepsilon$ be the Dirichlet character attached to $K$. If $f$ has CM and corresponds to a Hecke character $\psi \colon I(\mathfrak{m}) \to \mathbb{C}^\times$, then its level is $N = \mathrm{Norm}(\mathfrak{m}\mathcal{D})$, and its Nebentypus is given by $a \mapsto \varepsilon(a)\psi((a))/a$ for $(a,N)=1$. In the situation of Lemma~\ref{Nualart}, , the conductor $\gf(r)$ of the extension $\Qbar^{\ker  r}/K$ divides~$\ell\gm$. When the Nebentypus of $f$ is trivial, the ideal~$\cD$ must divide $\gm$. We notice that the conductor $\gf(r)$ may divide $\ell\gm \cD^{-1}$. See the examples in Section~8.2.4.
\end{remark}

The following result is a reformulation of \cite[Theorem~1.1]{Billerey} and provides supporting evidence for Conjecture~\ref{conj}.

\begin{theorem}[Billerey--Mortarino]
Let $\ell\geq 5$ and let $f\in S_2(\Gamma_0(N))$ be a newform of CM type modulo $\mathfrak{L}$ by an imaginary quadratic field $K$. Let $\varrho$ be the Galois representation associated with $f$ and $\mathfrak{L}$. If the image of the projective representation $\overline{\varrho}$ is dihedral and the Serre weight $k(\varrho)$ satisfies $2\leq k(\varrho)\leq \ell-1$, then $f$ is congruent to a newform $g\in S_{k(\varrho)}(\Gamma_0(M))$ with CM by $K$ and level
\[
M =
\begin{cases}
N(\varrho) & \text{if $\ell$ is unramified in $K$,} \\
\ell^2 N(\varrho) & \text{if $\ell$ ramifies in $K$.}
\end{cases}
\]
\end{theorem}

\begin{remark}
Conjecture~\ref{conj} predicts that such a congruence exists already in weight $2$, even when $k(\varrho)>2$ or the projective image is not dihedral, and for all primes $\ell\geq 2$.
\end{remark}

%---------------
\section{Hecke deformations of linear characters}

We begin by recalling that a Hecke character (of type $(1,0)$) modulo an ideal $\mathfrak{m}$ 
of the ring of integers $\mathcal{O}$ of $K$ is a group homomorphism
\[
\psi\colon I(\mathfrak{m})\longrightarrow \mathbb{C}^*
\]
such that $\psi((\alpha))=\alpha$ for all $\alpha \equiv 1 \pmod{\mathfrak{m}}$. Associated to~$\psi$ there is a Dirichlet character 
\[
\eta\colon (\mathcal{O}/\mathfrak{m})^*\longrightarrow \mathbb{C}^*
\]
satisfying $\eta(u)=1/u$ for all $u\in \mathcal{O}^*$ and such that
\[
\psi((\alpha))=\alpha\,\eta(\alpha)
\]
for all $\alpha\in\mathcal{O}$ coprime with $\mathfrak{m}$.

For such a Dirichlet character $\eta$ modulo $\mathfrak{m}$, 
there are exactly $h(K)$ Hecke characters $\psi$ modulo~$\mathfrak{m}$ satisfying $\psi((\alpha))=\alpha\eta(\alpha)$, where $h(K)$ is the class number of $K$.

Now let 
\[
r\colon \Gal(\overline{\mathbb{Q}}/K)\longrightarrow \mathbb{F}_{\ell^n}^*
\]
be a linear character. By class field theory, there exists an ideal $\mathfrak{m}$ supported on the primes ramified in the extension $\overline{\mathbb{Q}}^{\ker r}/K$ such that $r$ can be viewed as a morphism
\[
r\colon I(\mathfrak{m})\longrightarrow \mathbb{F}_{\ell^n}^*
\]
factoring through $I(\mathfrak{m})/H$, where
\[
H \supseteq P_{1}(\mathfrak{m})=
\{ (\alpha) \in I(\mathfrak{m}) \colon \alpha\equiv 1 \bmod \mathfrak{m} \}
\]
is a congruence subgroup. In particular, if $(\alpha)$ satisfies $\alpha\equiv 1 \pmod{\mathfrak{m}}$, then $r((\alpha))=1$.

\begin{definition}
We say that a morphism $r\colon I(\mathfrak{m})\longrightarrow \mathbb{F}_{\ell^n}^*$ satisfies the \emph{Hecke property} $(H)$ if:
\begin{itemize}
    \item[(i)] there exists a prime ideal $\mathfrak{l}\mid \ell$ of $K$ such that $(\mathcal{O}/\mathfrak{l})^*$ embeds into $\mathbb{F}_{\ell^n}^*$;
    \item[(ii)] $\mathfrak{l}\mid \mathfrak{m}$;
    \item[(iii)] $\mathcal{O}^*$ embeds into $\mathbb{F}_{\ell^n}^*$;
    \item[(iv)] $P_{1}(\mathfrak{m})\subseteq \ker r$.
\end{itemize}
\end{definition}
    
Conditions (i) and (ii) can always be satisfied by replacing $n$ with $2n$ and $\mathfrak{m}$ with $\mathfrak{m}\mathfrak{l}$ if necessary. Condition (iii) excludes the cases $\ell=2$ and $\ell=3$ when $K=\mathbb{Q}(\sqrt{-3})$. Condition (iv), which requires $\ker r$ to be a congruence subgroup, is the most substantial one, although it can also be ensured in practice as we shall discuss.

\begin{theorem}
\label{deformation}
Let $K/\mathbb{Q}$ be an imaginary quadratic field and 
let $r\colon I(\mathfrak{m})\longrightarrow \mathbb{F}_{\ell^n}^*$ be a linear character satisfying property $(H)$. 
Then there exists a prime ideal $\overline{\mathfrak{L}}\mid \ell$ in $\overline{\mathbb{Q}}$ and a Hecke character 
\[
\psi\colon I(\mathfrak{m})\longrightarrow \overline{\mathbb{Q}}^*
\]
such that
\[
\psi(\mathfrak{p})\equiv r(\mathfrak{p}) \pmod{\overline{\mathfrak{L}}}
\]
for all prime ideals $\mathfrak{p}$ in $I(\mathfrak{m})$.
\end{theorem}

\begin{proof}
Conditions (ii) and (iv) imply that the map
\[
\eta\colon (\mathcal{O}/\mathfrak{m})^*\longrightarrow \mathbb{F}_{\ell^n}^*, 
\quad 
\eta(\alpha)= r((\alpha))\,(\alpha \bmod \mathfrak{l})^{-1}
\]
is a well-defined character satisfying $\eta(u)=(u \bmod \mathfrak{l})^{-1}$ for all $u\in \mathcal{O}^*$.

Let $M=K\cdot \mathbb{Q}(\zeta_{\ell^n-1})$, and let $\overline{\mathfrak{L}}$ be a prime of $M$ above the prime $\mathfrak{l}$ of $K$ attached to $r$ by condition (i). 
Denote by $\mathfrak{L}=\overline{\mathfrak{L}} 
\cap \Z[\zeta_{\ell^n-1}]$ 
the corresponding prime ideal in the 
$\ell^n-1$\,th cyclotomic field.
Consider $\omega_{\mathfrak{L}}\colon\F_{\ell^n}^*\longrightarrow \Z[\zeta_{\ell^n-1}]^*$ to be a Teichmüller character attached to $\mathfrak{L}$; that is, it holds 
    $\omega_{\mathfrak{L}}(x)^{\ell^n}=\omega_{\mathfrak{L}}(x)$ and 
    $\omega_{\mathfrak{L}}(x) \bmod \mathfrak{L}$ is equal to $x$ for all $x$ in $\F_{\ell^n}^*$. 
Let $\widetilde \eta$ denote the Teichmüller lifting of 
the character $\eta$: 
    $$
\begin{tikzcd}
 &   \Z[\zeta_{\ell^n-1}]^*  \\
 &                   \\
 (\mathcal{O}/
 \mathfrak{m})^* \arrow{r}{\eta} \arrow{ruu}{\widetilde \eta}   & 
\F_{\ell^n}^* 
    \arrow[uu,"\ \omega_{\mathfrak{L}}",swap] \,.
\end{tikzcd}
$$
For $u\in \cO^*$, one has $\widetilde\eta (u)=1/u$ since $u$ is the unique unit in $\cO^*$ that reduces mod $\overline{\mathfrak{L}}$
to $u\pmod{\ell}$ due to the fact that $r$ satisfies condition (iii). Moreover, it holds 
$\widetilde\eta(\alpha)\equiv r((\alpha))\,
    \alpha^{-1} 
    \bmod \overline{\mathfrak{L}} $  for all $\alpha\in \mathcal{O}$ coprime
    with $\mathfrak{m}$.

Now, we consider generators of the class group
$
\operatorname{Cl}(K)=\langle \mathfrak{p}_1\rangle\times \dots \times
    \langle\mathfrak{p}_t\rangle
    $
    with $\mathfrak{p}_i$ in $I(\mathfrak{m})$. 
    Write $e_i=\operatorname{ord}(\mathfrak{p}_i)$,
    $\mathfrak{p}_i^{e_i}=(\alpha_i)$.
    Notice that 
for every $i=1,\dots,t$, the polynomial
$x^{e_i}-\alpha_i \widetilde\eta(\alpha_i)
\in \Qbar[x]$ reduced mod $\overline{\mathfrak{L}}$ is
$x^{e_i}-r((\alpha_i))=
x^{e_i}-r(\mathfrak{p}_i)^{e_i}
\in \F_{\ell^{n}}[x]$. 
We choose 
an element $a_i$ in $\Qbar$ such that
$a_i^{e_i}=\alpha_i\widetilde\eta(\alpha_i)$ and $a_i$ reduces to
$r(\mathfrak{p}_i)$ mod $\overline{\mathfrak{L}}$.
We define $\psi(\mathfrak{p}_i)=a_i$.
Also, for every $\alpha$ in $K$ coprime with $\mathfrak{m}$, we define
$\psi((\alpha)) = \alpha \widetilde\eta(\alpha)$.
For a prime  ideal $\mathfrak{p}$ in $I(\mathfrak{m})$,
    we can write:
    $$
    \mathfrak{p} = 
    \prod_{i=1}^t 
    \mathfrak{p}_i^{n_i}
    \cdot 
    (\alpha)\,,
    $$ 
    where $n_i\geq 1$, and $\alpha$ in $K$ with $(\alpha,\mathfrak{m})=1$. We define multiplicatively 
    $\psi$ over $I(\mathfrak{m})$
    and one has the congruence
    $$
    \psi(\mathfrak{p})=
    \prod_{i=1}^t 
    \psi(\mathfrak{p}_i)^{n_i}
    \cdot 
    \alpha \, \widetilde\eta(\alpha) \equiv
    \prod_{i=1}^t 
    \psi(\mathfrak{p}_i^{n_i})
    \cdot 
    r((\alpha))
    \equiv
    \prod_{i=1}^t 
    r(\mathfrak{p}_i^{n_i})
    \cdot 
    r((\alpha)) \equiv r(\mathfrak{p}) \bmod{\overline{\mathfrak{L}}} .
    $$
We can conclude that since $r$ is a morphism and satisfies (H), then 
$\psi$ is a Hecke character mod 
$\mathfrak{m}$ and 
$\widetilde\eta(\alpha)=\psi((\alpha))/\alpha$ is its associate Dirichlet character.
\end{proof}

\begin{remark}
If $r((p))=p\,\varepsilon(p)$ for all primes $p$ coprime with $\mathfrak{m}$, then the Nebentypus of the modular form attached to $\psi$ is trivial.
\end{remark}

\begin{theorem}
\label{teor}
Let $f\in S_2(\Gamma_0(N))^{\operatorname{new}}$ be a newform of CM type mod $\ell$ by $K$. Assume that $\ell>2$ and $\ell\neq 3$ when $K=\mathbb{Q}(\sqrt{-3})$. Then $f$ satisfies Conjecture~1.2 for some level $M$.
\end{theorem}

\begin{proof} Let 
$
\varrho\colon 
\Gal(\Qbar/\Q)\longrightarrow \GL_2(\overline{\F}_\ell)$ be the Galois representation
attached to $f$ and the corresponding 
prime ideal $\mathfrak{L}\mid \ell $ of its Fourier number field. Let $\rho^{\sss}$ be the semisimplification of $\rho$. By Lemma 6.7 in \cite{BP}, $\rho^{\sss}$ is the induced representation of a character $r\colon \Gal(\Qbar/K)\longrightarrow \F_{\ell^n}^*$.
    Let $\mathfrak{f}$ be the conductor of the abelian
    extension $\Qbar^{\ker r}/K$. By Theorems 8.3 and 8.5 in \cite{Cox}, 
    for every $\mathfrak{m}$ such that $\mathfrak{f}\mid \mathfrak{m}$, the character $r$ can be viewed as a morphism $r\colon I(\mathfrak{m}) \longrightarrow \F_{\ell^{n}}^*$ whose kernel is a congruence subgroup. Under the hypothesis on $\ell$, $r$ satisfies the property (H)
    with $\mathfrak{m}=\mathfrak{f}$ or $\mathfrak{m}=
    \ell \mathfrak{f}$ (in case that 
    $\ell$  and  $\mathfrak{f}$ are coprime).
    The Hecke character $\psi$ attached to $r$ by Theorem \ref{deformation} 
    yields a CM cuspidal eigenform. Its associated newform provides the required congruences with $f$ since $\rho$ and $\rho^{\sss}$ have the same traces and determinants.
\end{proof}

\begin{remark}
As we shall see in Section~\ref{examples},
the data of $r$, its conductor $\mathfrak{f}$, and the conductor of the associated Dirichlet character allow one to determine explicitly the level $M$ of the CM newform congruent to $f$. We conjecture that $M$ divides $N\ell$ for $\ell>2$.
\end{remark}

\begin{remark}
An earlier version of the above Theorem \ref{teor}  required stronger assumptions. We thank 	Aurel Page for pointing out Lemma 6.7 in his preprint \cite{BP}, which not only simplified our requirements on $f$  but also yielded the following generalization whose proof follows the same arguments as in Theorem \ref{teor}.  
\end{remark}

\begin{theorem} \label{thm37}
Let $\ell>2$ and assume $\ell\neq 3$ when $K=\mathbb{Q}(\sqrt{-3})$.  
Let $r\colon \Gal(\overline{\mathbb{Q}}/K)\longrightarrow \mathbb{F}_{\ell^n}^*$ be a linear character. Then the induced representation
$
\varrho=\operatorname{Ind}_K^{\mathbb{Q}}(r)
$
arises from a weight~2 newform with CM by $K$.
\end{theorem}
%--------------
\section{Abelian surfaces with extra endomorphisms}
\label{many_endo}

Let $A/\Q$ be a modular abelian surface attached to a newform 
$f\in S_2(\Gamma_0(N))^{\operatorname{new}}$ without complex multiplication. 
Let $E=\Q(\sqrt{m})$ be the Fourier number field of $f$, with $m$ square-free. 
By \cite{ribet1992taejon}, 
the abelian variety $A$ has extra endomorphisms (that is, 
$\operatorname{End}_\Q(A)\subsetneq \operatorname{End}_{\overline{\Q}}(A)$) 
if and only if $f$ admits a quadratic inner twist:
$$
f^\sigma = f\otimes \chi_\sigma,
$$
where $\sigma$ denotes the non-trivial automorphism of $E$ and 
$\chi_\sigma$ is the quadratic character attached to a quadratic number field 
$K=\Q(\sqrt{\delta})$, for some square-free~$\delta$. 
There are two possibilities for $A$:

\begin{itemize}
    \item[(i)] $A/\overline{\Q}$ is simple. This occurs if and only if $m$ is not a norm from $K$. In this case, $K$ is an imaginary quadratic field and $\operatorname{End}_{\overline{\Q}}(A)$ is an order in the indefinite quaternion algebra $B(m,\delta)_\Q$, whose discriminant $D$ is a multiple of $m$ (cf. Theorem~1.4 in \cite{BFGR}).

    \item[(ii)] $A/\overline{\Q}$ is not simple. In this case, $A$ is isogenous to the Weil restriction of a $\Q$-curve $\mathcal{E}$ completely defined over the quadratic (real or imaginary) field $K$. That is, the elliptic curve $\mathcal{E}$ is isogenous to its Galois conjugate over $K$.
\end{itemize}

\noindent In both cases, all endomorphisms of $A$ are defined over $K$. The Fourier coefficients of the normalized $q$-expansion 
$f(q)=\sum_{n=1}^\infty a_n q^n$ satisfy:
$$
\begin{array}{ll}
a_p \in \Z & \text{if $p$ splits in $K$,}\\[4pt]
a_p=c_p\sqrt{m},\ \ c_p\in \Z 
& \text{if $p$ is inert in $K$.}
\end{array}
$$

\begin{definition}
A prime $\ell$ is called \emph{exceptional} for the newform $f$ if $\ell \mid d\,m$, 
where $d\geq 1$ is determined by the relation 
$\Z[d\sqrt{m}]=\Z[\{a_n\}]$. 
(Note that $dm$ is not necessarily the discriminant of the order $\Z[\{a_n\}]$.) 
\end{definition}

\noindent Observe that for every exceptional prime $\ell$ for $f$, one has 
$$
a_p \equiv 0 \pmod{\ell}
$$
for all primes $p\nmid \ell N$ that are inert in $K$. 
Therefore, when $K$ is an imaginary quadratic field, 
the newform $f$ is of CM type modulo~$\ell$. 

The following result describes a restriction on $K$ when the Galois representation attached to $f$ at an exceptional prime $\ell$ is monomial but not dihedral.

\begin{proposition}
Let $f$, $E$, and $K$ be as above. Let $\ell>2$ be an exceptional prime for $f$. Assume that the projective semisimplification of the Galois representation $\varrho$ attached to $f$ and a prime $\mathfrak{L}\mid \ell$ of $E$ is not dihedral. Then $K$ is the unique quadratic subfield of the cyclotomic field $\Q(\zeta_\ell)$.
\end{proposition}

\begin{proof}
We claim that $K=\Q(\sqrt{\ell^*})$, where $\ell^*=(-1/\ell)\,\ell$ and 
$(-1/\ell)$ denotes the Legendre symbol. 
We know that $\varrho^{\mathrm{ss}}$ is isomorphic to $\operatorname{Ind}_K^\Q(r)$ for a linear character 
$r\colon \Gal(\overline{\Q}/K)\to \overline{\F}_\ell^*$. 
Since the projective representation of $\varrho^{\mathrm{ss}}$ is not dihedral, 
it follows from \cite{Nua} that $r=\overline{r}$. 
This implies that $a_p^2 \equiv 4p \pmod{\ell}$ for all primes $p\nmid N$ that split in $K$. Hence the Legendre symbol $(p/\ell)$ is trivial for such primes. 
If $\ell \equiv 1 \pmod{4}$, then $(\ell/p)=1$ for almost all $p$, which implies that $K=\Q(\sqrt{\ell})$. 
If $\ell \equiv 3 \pmod{4}$, then $(-\ell/p)=1$ for almost all $p$, and hence $K=\Q(\sqrt{-\ell})$. 
\end{proof}

%---------------------------
\section{Isogeny characters and QM-abelian surfaces}

 In this Section, suppose that the modular abelian surface $A/\Q$ attached to  the newform $f$ has  quaternionic multiplication. As before, $E=\Q(\sqrt{m})$ is the Fourier number field of $f$, $K=\Q(\sqrt{\delta})$ is a quadratic imaginary field,
 and $m$ divides the discriminant~$D$ of the quaternion algebra
 $B(m,\delta)_\Q$. 

 As a consequence of the discussion in Section 4 in \cite{Jordan}, we get the following result.
 
 \begin{proposition}
 \label{Jordan}
  Assume that $\End_K(A)$ is a maximal order in the quaternion algebra $ \End_K^0(A)=\End_K(A)\otimes\Q$. If a prime $\ell\mid D$ is exceptional for $f$, then there is a linear character $r\colon G_K\longrightarrow \F_{\ell^2}^*$ such that the induced representation $\operatorname{Ind}_{K}^{\Q} r$ is isomorphic to the representation
 $\varrho$ attached to  $f$  over  a prime $\mathfrak{L}\mid\ell$ of the Fourier number field of $f$.
 \end{proposition}
 
  \begin{proof}
      By Proposition 4.4 in \cite{Jordan}, if $\ell\mid D$ then there is a subgroup $C$ of $A[\ell]$ of order $\ell^2$ that is rational over $K$ and  $\End_K(A)$-stable. The group $C$ is called the canonical subgroup of $A$ at $\ell$ and  the Galois character  $r\colon G_K\rightarrow \F_{\ell^2}^*$ associated  with $C$ is called the isogeny character at $\ell$.
            In  Proposition~4.6 \cite{Jordan}, it is proved that $\operatorname{Norm}_{\F_{\ell^2}/\F_{\ell} }(r (\Frob_\mathfrak{p}))=\operatorname{Norm}( \mathfrak{p})$
      for every prime $\mathfrak{p}$. Hence, the determinant of the  induced representation $\operatorname{Ind}_{K}^{\Q} r$ is the cyclotomic character and, for every prime $p\nmid N\,\ell$, the trace of $\operatorname{Ind}_{K}^{\Q} r(\Frob_p)$ is $a_p\pmod \ell$ when $p$ splits in $K$ and $0$
when $p$ is inert in $K$. Thus, 
since $\ell$ is exceptional,  $\operatorname{Ind}_{K}^{\Q} r$ and $\varrho$ are isomorphic. \end{proof}

\begin{remark}
The linear character $r$
from Proposition \ref{Jordan} is called isogeny character in \cite{Jordan}. Indeed, due to the fact that $\End_K(A)$ is a maximal order, it has a unique bilateral ideal of norm $\ell$ such that its square is the ideal generated by $\ell$. Since the quaternion algebra is indefinite the ideal is principal and, hence, there is an isogeny $\varphi\colon A\longrightarrow A$ of degree $\ell^2$ whose kernel is $C$.
\end{remark}

\begin{remark}
Observe that the quaternionic surfaces  $A/K$ treated in \cite{Jordan}
are not necessarily modular. Nevertheless, for $\ell>2$ or 
$\ell\neq 3$ when $K=\Q(\sqrt{-3})$, by Theorem \ref{thm37}
we have that $\operatorname{Ind}_{K}^{\Q} r$ arises form a newform of weight 2 with CM by $K$ and, moreover,  it has trivial Nebentypus.
\end{remark}

Now,  under the modularity assumption on $A$, we want to show that we can get an isogeny character at $\ell$ relaxing  the condition on the order of the 
endomorphism ring of $A$ to be maximal.

First, we fix the following notation. According to \cite{ribet1992taejon}, we can consider $T,S\in \End_{K}^0(A)$  whose action on $\langle f,  f^\sigma\rangle$ is as follows:
  $$ T:(f,f^\sigma)\mapsto (\sqrt{m} f,-\sqrt{m} f^\sigma) \,,\,\,\quad S:(f,f^\sigma)\mapsto (\sqrt{\delta} f^\sigma,\sqrt{\delta} f).$$
It holds
$T^2=m$, $S^2=\delta$, $T\cdot S=-S\cdot T$, $\overline T=T$ and $\overline S=-S$.
We denote by $n_T$, resp. $n_S$, the minimum positive integer such that $n_TT$, resp. $n_SS$,  belongs to the ring $\End_K(A)$.

The following result 
will be useful
for our purpose. 
We are grateful to Santi Molina for the localizing argument that makes its proof immediate.

\begin{lemma} 
\label{Santi} 
Let $\cB$ be an indefinite   quaternion algebra   over $ \Q$ of discriminant $D$ and let  $L$ be a quadratic field. If $L$ is contained in $\cB$ or $\cB$  is contained in the matrix algebra  $\operatorname{M}_2(L)$, then every   prime $\ell\mid D$ does not split in $L$.
\end{lemma}

Since $A$ is a QM-abelian variety and
$K\subseteq \End^0(A)$, 
as a consequence of the above lemma
we can restrict ourselves to consider the cases $\ell\mid D$ where 
$\ell$ either ramifies  
or it is inert in  $K$ ($\ell\mid m$ in this second case).

\begin{proposition}
\label{n_t_n_s}
Let   $\ell$ be an inert prime in $K$ such that $\ell\mid m$ and $\ell \nmid n_T\cdot n_S$. If $\delta \not\equiv 1 \pmod 4$ or $\ell \neq 2$, the representation $\varrho$ attached to $f$ at $\ell$  is
the induced  representation of a linear  character $r\colon G_K\longrightarrow \F_{\ell^2}^*$.

\end{proposition}

\begin{proof}
Since $\ell \nmid n_T$ and $n_T T$ is an endomorphism of $A$,  the group  $\ker (n_TT)^2 \cap A[\ell]$ is isomorphic to  $(\Z/\ell\Z)^4$. Hence, the group $C:=\ker (n_T T) \cap A[\ell]$ is isomorphic to $(\Z/\ell\Z)^2$ and, moreover,  
$G_\Q$-stable. Since $T\cdot S=-S \cdot T$ and $\ell \nmid n_S$, then $S$ acts on~$C$ and, thus, one gets an action of $\Z[\sqrt{\delta}]$ on $C$.  If $\delta \not\equiv 1 \pmod 4$ or  $\ell\neq 2$, the ring of integers $\cO$ of  $K$ acts on $C$ and, thus, the field  $\cO/\ell\Z\simeq \F_{\ell^2}$ also acts on this group. Thus, $C$ is a $\F_{\ell^2}[G_K]$-module of dimension $1$. The character $r\colon G_K\rightarrow \F_{\ell^2}^*$  defined  by
$$\Frob_\mathfrak{p}(P)=r(\Frob_{\mathfrak{p}})P,  \,\, \text { for all } P \in  C,$$ induces the representation $\varrho$.     
\end{proof} 

\vskip 0.2 cm  

Observe that if there exists a prime
$p\nmid N$ inert in $K$ satisfying $\ell\nmid a_p^2/m$, then ~$\ell \nmid n_T$. 

\begin{proposition}
\label{55}
Let  $\ell\neq 2$ be an  exceptional prime dividing $D$ ramified in $K$ and such that $\ell\nmid n_S$.
Then
the representation $\varrho$ attached to $f$ and $\ell$ is
induced  by a character $r\colon G_K\longrightarrow \F_{\ell^2}^*$.
\end{proposition}

\begin{proof}
For every prime $p\nmid N$ that splits in $K$, we consider the order $\cL_p=\Z[(a_p+\sqrt{(a_p^2-4p)/2}]$ of the imaginary 
quadratic field $L_p=\Q(\sqrt{a_p^2-4p}) $. We claim that the prime $\ell$ does not split in $L_p$ and the polynomial $x^2-a_p x+p \pmod{\ell}$  has a unique root if and only if $\ell$ ramifies in $L_p$ or $\ell$ divides the conductor of $\cL_p$. Indeed,
$A\otimes \F_p$  is $\F_p$-isogenous to the square of an elliptic curve $\mathcal{E}  /\F_p$ since the  characteristic  polynomial of $\Frob_p$ acting on the Tate module  of $A$ is $(x^2-a_px+p)^2$ due to the 
Eichler-Shimura congruence. Hence, $L_p=\End_{\F_p} ^0(\mathcal{E})$ and  $\End^0(A)$ are both contained in the matrix algebra $\operatorname{M}_2(L_p)$. Therefore, $\ell$ does not split in  $L_p$ by Lemma \ref{Santi}.  The polynomial  $x^2-a_p x+p \pmod{\ell}$  has a unique root if and only if $\ell\mid (a_p^2-4p)$, i.e. $\ell$ ramifies in $L_p$, or $\ell$ is inert and divides the conductor of $\cL_p$.

Since $\ell \nmid n_S$, the group $C:=\ker (n_S S) \cap A[\ell]$ is isomorphic to $(\Z/\ell\Z)^2$ and it is $G_\Q$-stable since $\overline S=-S$.
We define the character $r\colon G_K\longrightarrow \F_{\ell^2}^*$ at the Frobenius automorphisms as follows. Let~$\mathfrak{p}$ be a prime of  $K$ over $p\nmid N$.  If $p$ is inert in $K$, we define
$r(\Frob_\mathfrak{p})=-p \pmod{\ell}$. If $p$ splits in $K$ and  $\ell$ ramifies in $L_p$ or $\ell$ divides the conductor of  $\cL_p$, we define $r(\Frob_\mathfrak{p})=r(\Frob_{\overline{\mathfrak{p}}})= a_p/2 \pmod{\ell}$. For the remaining cases, since $\ell$ does not divide the conductor of $\cL_p$, the integer ring $\cO_L$ of~$L$ acts on $C$. Thus, $C$ can be viewed as a 1-dimensional  $\cO_L/\ell\cO_L
\simeq \F_{\ell^2}$-vector space and we define $r(\Frob_\mathfrak{p})$ by
$\Frob_{\gp}(P)=r(\Frob_{\gp})P  \,\, \text { for all } P \in  C.$
In all cases, one readily checks that the character~$r$ induces the representation~$\varrho$.   
\end{proof}

%-------------------
\section{Isogeny characters and $\Q$-curves}

For general properties of the isogeny characters attached to cyclic isogenies of elliptic curves, we refer the reader to Sections~5 and~6 of \cite{Mazur}.  
We keep the notation of Section~\ref{many_endo}.

\begin{proposition}\label{weil}
Assume that the modular abelian surface $A$ attached to the newform 
$f(q)=\sum_{n=1}^\infty a_n q^n\in S_2(\Gamma_0(N))^{\operatorname{new}}$ 
is $\Q$-isomorphic to the Weil restriction of a $\Q$-curve $\mathcal{E}$ defined over a quadratic field $K$ (real or imaginary) of degree $m$.  
Let $\ell\mid m$ be an exceptional prime for $f$, and let $\mathfrak{l}$ be a prime ideal of $\Q(\sqrt{m})$ above $\ell$. Then:

\begin{itemize}
\item[(i)] The representation $\varrho_{f,\mathfrak{l}}\colon G_\Q\longrightarrow \GL_2(\F_{\ell})$ is isomorphic to an induced representation $\operatorname{Ind}_K^{\Q}(r)$, where $r\colon G_K\longrightarrow \F_{\ell}^*$ is a linear character.

\item[(ii)] For every prime $p\nmid N$ that splits in $K$, the polynomial $x^2-a_p x+p$ is reducible modulo $\ell$.
\end{itemize}
\end{proposition}

\begin{proof}
Let $p\nmid N$ be a prime. Denote by $\Frob_p$ and $\Ver_p$ the Frobenius and Verschiebung endomorphisms acting on the reduction $A\otimes \F_p$. One has $\Frob_p\circ \Ver_p = p$.  
By the Eichler--Shimura congruence relation, the reduction of the Hecke operator $T_p$ modulo $p$ is given by $\Frob_p+\Ver_p$.

Let $\phi\colon \mathcal{E}\longrightarrow {\,}^\sigma \mathcal{E}$ be the isogeny of degree $m$ defined over $K$, where $\sigma$ denotes the non-trivial automorphism of $K$. The group
\[
C:=\ker\phi\cap \mathcal{E}[\ell]
\]
is isomorphic to $\Z/\ell\Z$ and is stable under the action of $G_K$. This action defines the isogeny character 
\[
r\colon G_K\longrightarrow \F_\ell^*,
\]
given by the rule that for a prime ideal $\mathfrak{p}$ of $K$ above $p\nmid N$,
\[
\Frob_{\mathfrak{p}}(P)=r(\Frob_{\mathfrak{p}})\,P \qquad \text{for all } P\in C.
\]

If $p$ splits in $K$, then the set $\{\Frob_{\mathfrak{p}},\Frob_{\overline{\mathfrak{p}}}\}$ coincides with $\{\Frob_p,\Ver_p\}$. If $p$ is inert in $K$, then $\Frob_{\mathfrak{p}}=\Frob_{p^2}$. Therefore,
\[
\operatorname{tr}\bigl(\operatorname{Ind}_K^{\Q}(r)(\Frob_p)\bigr)
= r(\Frob_{\mathfrak{p}}) + \varepsilon(p)\, r(\Frob_{\overline{\mathfrak{p}}})
\equiv a_p \pmod{\ell},
\]
and
\[
\det\bigl(\operatorname{Ind}_K^{\Q}(r)(\Frob_p)\bigr)=p.
\]
This proves (i). Statement (ii) follows from the fact that $r(\Frob_{\mathfrak{p}})$ is a root of $x^2-a_p x+p$ modulo~$\ell$.
\end{proof}

\begin{remark}
The conductor of the modular abelian surface $A$ is $N^2$ (see \cite{Carayol}). On the other hand, the conductor of the $\Q$-curve $\mathcal{E}$ is stable under Galois conjugation, and hence can be written in the form 
\[
a\cdot \prod_{\mathfrak{p}} \mathfrak{p},
\]
where $a$ is an integer and the product runs over the (possibly empty) set of prime ideals of $K$ that ramify.
By \cite{Milne}, the conductor of $A$ is equal to the norm of the conductor of $\mathcal{E}$ times $\Delta_K^2$, where $\Delta_K$ denotes the discriminant of $K$. It follows that the above product is empty, and hence the conductor of $\mathcal{E}$ is the ideal generated by $N/\Delta_K$.
\end{remark}

%-------------------------------------------
\section{Elliptic curves over $\Q$ of CM type mod $5$}\label{ECmod5}

In this section, we focus on a family of elliptic curves over $\Q$ whose mod~$5$ Galois representation has image isomorphic to a specific subgroup $G$ of $\GL_2(\F_5)$.

More precisely, we consider the modular maximal cyclic group $G$ of order $16$, namely
$G = M_4(2) = [16,6] = 8T7$ (for more details on this group, see the GroupNames database:
\texttt{https://people.maths.bris.ac.uk/$\sim$matyd/GroupNames/1/M4(2).html}).
A presentation of $G$ is given by
\[
G = \langle a,b \mid a^8 = b^2 = 1,\; bab = a^5 \rangle.
\]

Up to conjugacy, there is a unique subgroup of $\GL_2(\F_5)$ isomorphic to $G$. We fix the following generators:
\[
a = \begin{pmatrix} 0 & 1 \\ 2 & 0 \end{pmatrix}, 
\qquad
b = \begin{pmatrix} 4 & 0 \\ 0 & 1 \end{pmatrix}.
\]
The lattice of subgroups of $G$ is explicitly described by the following diagram:
\begin{center}
$$
\begin{tikzpicture}[scale=1.2,sgplattice]
  \node[char] at (2,0) (1) {\gn{C1}{1}};
  \node[char] at (3.25,0.803) (2) {\gn{C2}{\langle a^4 \rangle}};
  \node at (0.75,0.803) (3) {\gn{C2}{\langle b \rangle \text{ and }\langle a^4b \rangle}};
  \node[char] at (2,1.89) (4) {\gn{C4}{\langle a^2 \rangle}};
  \node[char] at (0.125,1.89) (5) {\gn{C2^2}{\langle b \rangle{\times}\langle a^4 \rangle}};
  \node[char] at (3.88,1.89) (6) {\gn{C4}{\langle a^2 b \rangle}};
  \node[norm] at (0.125,3.11) (7) {\gn{C8}{\langle a \rangle}};
  \node[norm] at (3.88,3.11) (8) {\gn{C8}{\langle ab \rangle}};
  \node[char] at (2,3.11) (9) {\gn{C2xC4}{\langle b \rangle{\times}\langle a^2 \rangle}};
  \node[char] at (2,4.06) (10) {\gn{M4(2)}{\langle a,b \rangle}};
  \draw[lin] (1)--(2) (1)--(3) (2)--(4) (2)--(5) (3)--(5) (2)--(6) (4)--(7)
     (4)--(8) (4)--(9) (5)--(9) (6)--(9) (7)--(10) (8)--(10) (9)--(10);
  \node[cnj=3] {};
\end{tikzpicture}
$$
\end{center}

The subgroups $\langle b \rangle$ and $\langle a^4 b \rangle$ are the two non-normal conjugate subgroups of $G$. The remaining quotients by normal subgroups can be easily deduced; alternatively, we refer to the GroupNames database.

By \cite{Zywina}, the non-cuspidal rational points of the modular curve $X_G(5)$ classify elliptic curves $E/\Q$ such that the image of the mod~$5$ Galois representation
$\varrho_{E,5}$ acting on $E[5]$ is isomorphic to $G$. The curve $X_G(5)$ has genus zero.

Zywina shows that a Hauptmodul for $X_G(5)$ can be expressed in terms of the Rogers–Ramanujan continued fraction
\[
r(\tau) :=
q^{1/5}
\cdot \frac{1}{1+}
\cdot \frac{q}{1+}
\cdot \frac{q^2}{1+}
\cdot \frac{q^3}{1+}
\cdot \frac{q^4}{1+} \cdots
\]

We have $\Q(X_G(5)) = \Q(t)$, where
\[
t = t(\tau) =
\frac{(-3+\phi)h(\tau)-5}{h(\tau)+(3-\phi)},
\]
with $\phi = (1+\sqrt{5})/2 \in \Q(\zeta_5)$ and $h(\tau) = r(\tau)^{-1} - r(\tau) - 1$.
The forgetful map $j : X_G(5) \to X(1)$ is given by
\[
j(t)=
\frac{5^4 t^3 (t^2+5 t+10)^3 (2 t^2+5 t+5)^3 (4 t^4+30 t^3+95 t^2+150 t+100)^3}
{(t^2+5 t+5)^5 (t^4+5 t^3+15 t^2+25 t+25)^5}.
\]

A short computation using \textit{Mathematica} (Version 14.0, Wolfram Research, Champaign, IL) and \textit{Magma} yields the following result; see \cite{repo2026} for the implementation.

\begin{proposition}\label{71}
Let $E = E(t)$ be an elliptic curve with $j(E)=j(t)$ for $t \in \Q$. The three quadratic subfields of $\Q(E[5])/\Q$ are
\[
K_1=\Q(\sqrt{-(3t^2+10t+15)}), \qquad
\Q(\sqrt{5}), \qquad
K_2=\Q(\sqrt{-5(3t^2+10t+15)}).
\]
The Galois representation $\varrho_{E,5}$ is isomorphic to both
$\operatorname{Ind}_{K_1}^{\Q}(r_1)$ and $\operatorname{Ind}_{K_2}^{\Q}(r_2)$, where
\[
\begin{array}{lll}
r_1 : \Gal(\Q(E[5])/K_1) \simeq \langle a \rangle \to \F_{25}^*, 
& r_1(a)=\eta, & \eta^2=2,\\
r_2 : \Gal(\Q(E[5])/K_2) \simeq \langle ab \rangle \to \F_{25}^*, 
& r_2(ab)=\eta', & \eta'^2=3=\eta^6.
\end{array}
\]
\end{proposition}

\begin{remark}
Note that $K_1$ and $K_2$ are imaginary quadratic fields for all $t \in \Q$.
\end{remark}

\begin{proposition}
Let $E=E(t)$ be as above. Then the weight~2 newform
$f_E(q)=\sum_{n>0} a_n q^n$ attached to $E$ is of CM type mod~$5$ simultaneously by two different imaginary quadratic fields.
\end{proposition}

\begin{proof}
For every prime $p \neq 5$ of good reduction, the Fourier coefficient $a_p$ coincides modulo $5$ with $\operatorname{tr}\varrho_{E,5}(\Frob_p)$. Since $\varrho_{E,5}$ is monomial with respect to the quadratic fields $K_1$ and $K_2$, it follows that $a_p \equiv 0 \pmod{5}$ whenever $p$ is inert in either $K_1$ or $K_2$.
\end{proof}

\begin{remark}
The proportion of elements in $G \simeq \operatorname{Im}(\varrho_{E,5})$ with trace zero modulo $5$ is $3/4$.
\end{remark}

\begin{proposition}
Let $E=E(t)$ be as above. Then $E$ has Kodaira type $II$ or $IV^*$ at $5$, and the associated Serre weight is $k(\varrho_{E,5})=10$.
\end{proposition}

\begin{sketchproof}
Fix a Weierstrass model of $E(t)$ with invariants $(c_4,c_6,\Delta)$ such that $j(E(t))=j(t)$. Note that $j(t)=j(5/t)$. Write $t=5^n u$ with $u \in \Z_5^\times$, and compute the $5$-adic valuations $(v_5(c_4),v_5(c_6),v_5(\Delta))$. Then apply~\cite{Pal} to determine the Kodaira type at $5$, and Serre’s recipe (or alternatively~\cite{Kraus}) to compute $k(\varrho_{E,5})$.
\end{sketchproof}

We observe that \cite{Billerey} cannot be applied here since Serre’s weight exceeds $4$. Nevertheless, Theorem~\ref{teor} applies to this family of modular forms. Moreover, each representation $\varrho_{E,5}$ arises from two distinct weight~$2$ CM newforms, since it can be induced from linear characters of two different imaginary quadratic fields.
%-----------------------------
\section{Examples}
\label{examples}

The following examples provided the initial motivation for Conjecture~\ref{conj}. While some of them fall outside the scope of Theorem~\ref{teor}, many are covered by it; we include all of them here primarily for illustrative purposes.

We organize the examples into three types. In the first type, we start from modular forms associated either with quaternionic surfaces or with abelian surfaces isogenous to the square of a $\mathbb{Q}$-curve defined over an imaginary quadratic field.

In the second type, we begin with explicit equations of quadratic $\mathbb{Q}$-curves defined over imaginary quadratic fields. These yield more extensive and computationally involved examples, as they give rise to modular forms of very high level and involve imaginary quadratic fields with large class number.

The third type arises from the $5$-torsion of certain elliptic curves defined over $\mathbb{Q}$.

In the first type of examples, we start with a newform $f$ of CM type modulo $\ell$ and determine a genuine CM newform $g$ such that the congruence $f \equiv g \pmod{\ell}$ holds. If the number of verified congruences between the corresponding Fourier coefficients is sufficiently large relative to the levels involved, we may regard the conjecture as numerically verified in those cases.
In practice, it suffices (assuming GRH) to check the congruences for Fourier coefficients $a_p$ for all primes
$p \leq 2\log^2(\ell N_f N_g)$.
See \cite{SerreChebo} and \cite{Edix} for a detailed discussion of this criterion.

In some cases, however, we have only verified the congruences up to a fixed bound (typically $p < 500$, and in some instances up to $p < 3000$), chosen independently of the Sturm bound and without attempting to ensure completeness even under GRH. Our aim is solely to provide supporting numerical evidence for the phenomenon.

\subsection{Modular QM-abelian surfaces from Quer's data base} 
By using Magma, 
Quer~\cite{Quer}
found all abelian surfaces $A/\Q$ with QM arising from newforms $f(q)=\sum_{n>0}a_nq^n $ of level $N\leq 7000$.
The following table summarizes the data of the 40 cases obtained by Quer.
$$
\begin{array}{c}
\begin{array}{|ccccc||}
N & d_0 & \delta &D& \# A\\
\hline \hline
243 & 6 & -3&6& 1\\[3 pt]
675 & 2 & -3& 6&2\\[3 pt]
972 & 18 & -3&6& 1\\[3 pt]
1323 & 6&  -3& 6&2\\[3 pt]
1568 & 7 & -1&14& 2\\[3 pt]
1568 & 3 & -1& 6&2\\[3 pt]
1849 & 6 & -43&6& 1\\[3 pt]
2592 & 24 & -1& 6&2\\[3 pt]
 \hline
\end{array}
\begin{array}{|ccccc||}
N & d_0 & \delta &  D&\# A\\\hline\hline
2592 & 3& -1&6& 2\\[3 pt]
2601 & 2 & -51& 6&1\\[3 pt]
2700 & 10 & -3& 10&2\\[3 pt]
3136 & 3& -1&6& 2\\[3 pt]
3136 & 7 & -1& 14&2\\[3 pt]
3888 & 6 & -3& 6&1\\[3 pt]
3888 & 18 & -3& 6&1\\[3 pt]
3969 & 15 & -7& 15&1\\[3 pt]
\hline
\end{array}
\begin{array}{|ccccc|}
N & d_0 & \delta &  D&\# A\\
\hline \hline
5184 &  3 & -1& 6& 2\\[3pt]
5184 &  24 & -1& 6&2\\[3pt]
5292 & 10&-3& 10&2\\[3pt]
5408 & 11&-1 & 22&2\\[3pt]
5408 & 3&-13 & 6&2\\[3pt]
6075 & 6 & -3&6& 1\\[3 pt]
6400 & 6 & -1& 6&4\\[3 pt]
   &       &     & &\\[3 pt]
\hline
\end{array}\\\\
\mbox{ Table I}
\end{array}
$$
In the table, $D$ denotes the discriminant of the quaternion algebra
$B(m,\delta)_{\Q} = \End(A)\otimes \Q$, while $d_0$ is the integer such that
$\Z[\sqrt{d_0}] = \Z[\{a_n\}]$. We write $m = d_0/d^2$ for the square-free part of $d_0$, and $\#A$ denotes the number of modular QM-abelian surfaces of level $N$ having the same value of $d$ and the same imaginary quadratic field $K=\Q(\sqrt{\delta})$.

We remark that all the imaginary quadratic fields $K$ appearing in the table have class number $h(K)=1$, except for $\Q(\sqrt{-51})$ and $\Q(\sqrt{-13})$, which have class number $h(K)=2$. Moreover, the class number of $\Q(\sqrt{-D})$ is equal to $2$ in all cases listed. According to a conjecture of Coleman, the set of pairs $(D,m)$ arising from modular QM-abelian surfaces is expected to be finite; see \cite{BFGR}. It may in fact be the case that $h(\Q(\sqrt{-D})) = 2$ always holds, and more generally that $h(K)\leq 2$ in all cases.
All exceptional primes $\ell$ not dividing $m$ satisfy $\ell \mid D$.

For the $61$ exceptional primes $\ell$ appearing in Quer’s data, we have verified that there exists a normalized newform $g \in S_2(\Gamma_0(M))^{\operatorname{new}}$ with CM by $K$ satisfying the predicted congruences for all primes $p<500$, in accordance with Conjecture~\ref{conj}. The condition $\dim A_g = h(K)$ is satisfied in $57$ cases.
For the remaining four cases,
\[
(N,\delta,\ell)=
(2592,-1,2),\ (5184,-1,2),\ (5408,-1,11),\ (3969,-7,5),
\]
one has $\dim A_g = 2h(K) = 2$.

The newforms $g$ were computed using Magma. In all cases, we observe that the level $M$ of $g$ divides the level $N$ of $f$, as shown in the table below.

$$
\begin{array}{|c|c|c|c|c|}
\delta& N&d_0&\ell & M\\\hline
-1&1568 & 7&7 &1568,1568\\
&1568 &3&3 &32,1568\\
&2592&24&3&32,288\\
&2592&3&3&32,288\\
&3136& 3&3& 64,3136  \\
&3136&7 &7 &3136,3136   \\
&5184 &24 & 3& 64,576 \\
&5184& 3 & 3&  64,576 \\
&6400 &6 & 2& 6400,6400 \\
&6400&6 &3& 6400,6400 \\ \hline\hline
-7 & 3969& 15 &3 & 49\\ \hline\hline
-13 & 5408& 3 &3 & 5408,5408\\ \hline\hline
-43 & 1849& 6 &2,3 & 1849\\ \hline\hline
-51 & 2601& 2 &2 & 2601\\ \hline
\end{array}
\quad\quad
\begin{array}{|c|c|c|c|c|}
\delta& N&d_0&\ell & M\\ \hline
-3&243 & 6&2 &243\\
 & 243 &6&3 &27  \\
& 675 &2& 2& 27,27 \\
 &972 &18& 2 & 972\\
 &972 &18 &3 &27\\
 &1323 &6& 2 & 1323,1323\\
 &1323 &6& 3& 27,1323\\
 &2700&10&2 &27,27\\
 &2700&10&5& 2700,2700\\
 &3888&6,18&2&972  \\
 &3888&6,18&3& 3888  \\
 &5292&10&2& 5292,5292\\
 &5292&10 &5& 5292,5292\\
 &6075 &6& 2& 6075\\
 &6075& 6& 3 &6075 \\  \hline
\end{array} 
$$
\begin{center}
 Table II  
\end{center}

Notice that Theorem \ref{teor} cannot be applied to some of the examples above. For instance, in the case $N=243$ neither applies for $\ell=2$ nor $\ell=3$. Nevertheless, these examples also satisfy the congruences predicted for the Conjecture \ref{conj} for primes $p<500$.

%\newpage

\subsection{$\Q$-curves  completely defined over imaginary quadratic fields} 
As usual, we say  that a $\Q$-curve is completely defined over a quadratic field if the elliptic  curve and the isogeny to its Galois conjugate are defined over that quadratic field.
\subsubsection{First examples.}
As a first test, we have checked Conjecture \ref{conj} for all $\Q$-curves defined over imaginary  quadratic fields without CM
arising from  newforms $f(q)=\sum_{n>0}a_nq^n$ of level $N \leq 500$.

$$
\begin{array}{c}

\begin{array}{|ccccr|}
N & d_0 & \delta &\ell &M \\
\hline \hline
63 &12&-3& 2& 144\\[3 pt]
 63  &  12&-3 &3  &  27\\[3 pt]
81 & 3& -3&  3&27\\[3 pt]
98 & 2& -7 &2 &49\\[3 pt]
117 & 12&-3 & 2 &144\\ [3 pt]
 117  &  12  &  -3  & 3 & 27\\[3 pt]
160 & 8 &-1 &2 &32\\ [3 pt]
189 & 3&-3& 3&27\\[3 pt]
189 & 7& -3& 7& 27\\[3 pt]
243 & 12&-3& 2&3888\\[3 pt]
  243 & 12  &  -3 & 3& 27\\[3 pt] \hline
\end{array}
\qquad
\begin{array}{|ccccr|}
N & d_0 & \delta &\ell & M\\
\hline \hline
320 & 8& -1 & 2&32\\ [3 pt]
363 & 3& -11& 3 &121\\[3 pt]
363 & 5& -11 &5 &121\\[3 pt]
387 &3&-3 &  3&27\\[ 3 pt]
392 &8 &-7&  2&49\\[ 3 pt]
392 &2 &-7&  2&49\\[ 3 pt]
416 &  5&-1& 5 &32\\[3 pt]
441 &  12&-3& 2 &144\\[3 pt]
441  & 12 & -3 &3 & 1323\\[3 pt]
484& 3&-11 & 3 &121\\[ 3 pt]
    &     & & & \\[3 pt]
\hline
\end{array}

\\\\
\mbox{ Table III}
\end{array}
$$

From the above table,  $K=\Q(\sqrt{\delta})$ denotes the imaginary quadratic field where the corresponding $\Q$-curve is completely defined,
$d_0$ is  the integer such that $\Z[\sqrt {d_0}]=\Z[\{a_n\}]$, $m=d_0/d^2$ is the square-free part of $d_0$ so that $m$ is the degree 
of the $\Q$-curve and $M$ is the level of a newform $g$ such that $f \equiv g \pmod{\ell}$ for all primes $<500$.
For all the cases in the above table, one has $h(K)=1$. But, in contrast with the QM scenario, it turns out that the class number $h(K)$ can be quite large as shown in the following subsections.

\subsubsection{Second  examples}
For instance, there are 8 examples with $h(K) >2$ in the range $N\leq 2000$. The data $(N,m,\delta, h(K))$ for the corresponding modular abelian surfaces is given by: 
$$(1058,2,-23,3),(1058,3,-23,3),(1058,3,-23,3),(1587,2,-23,3),$$
$$ (1587,3,-23,3),(961,2,-31, 3) ,(1922,10,-31,3),  (1521,3,-39,4).$$
In these eight cases, we  have also checked Conjecture \ref{conj}. Indeed, one has $\operatorname{disc}\,(K)=\delta$ and there is a normalized newform
$g=\sum_{n>0} b_n q^n \in S_2(\Gamma_0(\delta^2))^{\operatorname{new}} $
 with CM by $K$ attached to a Hecke character~$\psi$ of  conductor $\mathfrak{{m}}=(\sqrt{\delta})$ and the  corresponding Dirichlet character $\eta$ is the quadratic character of~$(\cO/\sqrt{\delta}\cO)^*$.

Defining polynomials $P_{|\delta|} $ of the corresponding Fourier fields $E_g$ for each CM newform $g$ are:
$$P_{23}=x^3 - 6\,x - 3\,,\,\,P_{31}(x)=x^3 - 6\,x -1\,,\,\,P_{37}(x)=x^4 - 8\,x^2 + 3\,.
$$
It holds:
\begin{itemize}
\item $P_{23}\equiv(1 + x) (1 + x + x^2) \pmod 2\,,\,\, P_{23}=x^3 \pmod 3$
\item $P_{31}\equiv (1 + x) (1 + x + x^2)\pmod 2\,,\,\,P_{31}\equiv (3 + x) (3 + 2 x + x^2) \pmod 5$
    \item $ P_{39}\equiv x^2 (1 + x^2)\pmod 3\,.$
\end{itemize}

\noindent Hence, for every prime $\ell\mid  m$, there exists a unique   prime ideal $\mathfrak{l}\mid \ell$ of $E_g$  of  residue degree $1$ such that   $f\equiv g \pmod{ \mathfrak{l}}$.

\subsubsection{Third  example} The following example corresponds to a quadratic $\Q$-curve of degree 137 defined over an imaginary quadratic field with class number 117 (cf. Proposition 4.2 in \cite{Pep}). Among 
the $\Q$-curves completely defined over imaginary quadratic fields, this one has the largest known isogeny degree.
More precisely, let $\mathcal{E}=x^3+A\,x+B$ be the elliptic curve defined by
$$\begin{array}{ccl}
A&=&1713745(2643250204357 - 285242082633 \,\sqrt{\delta})/2\,,\\[4pt]
 B&=&5452825 (-18224167668804803284533 +
    63802091292233830777 \,\sqrt{\delta})\,,
    \end{array}$$ 
\noindent where $\delta=-31159$. One has that $\mathcal{E}$ is without CM and has an isogeny to its Galois conjugate of degree $137$  defined over  $K= \Q(\sqrt{\delta})$. Its  conductor is the  ideal of $K$ generated by $2\cdot 5^2\cdot 31159$. The  surface $A=\operatorname{Res}_{K/\Q}(\mathcal{E})$ satisfies
$\End_{\Q}^0(A)\simeq\Q(\sqrt{137})$ and, thus,
$A$ is ${\Q}$-isogenous to the modular abelian surface $A_f$, where $f=\sum_{n>0} a_n q^n\in S_2(\Gamma_0(2\cdot 5^2\cdot 31159^2))^{\operatorname{new}}$ is a normalized newform of CM type mod 137.

For every prime $p\nmid 2\cdot 5\cdot 31159$ splitting in $K$, we denote by $a_p=\operatorname{tr}(\operatorname{Frob}_p) $ the trace of the geometric Frobenius $\operatorname{Frob}_p$ acting on the reduction $\mathcal{E}\otimes\F_p$. For instance,

$$\begin{array}{c|r|r|r|r|r|r|r|r|}
p&11& 17& 19& 23& 29& 41& 43& 47\\\hline
a_p&0& -6&-2&6&-2&3& 4 &6\\ \hline
\end{array}
$$

\vskip 0.5truecm

\noindent The class group of $K$ is cyclic of order $117$ generated by one of the  prime ideals
 ${\mathfrak{p}}_{11}$ lying over~$11$. 
 Write ${\mathfrak{p}}_{11}^{117}=(\gamma)$ with $\gamma$ being a square mod $137$.
 To test Conjecture \ref{conj} in this case, we consider 
 the CM normalized newform $g=\sum_{n>0} b_n q^n$ attached to the following Hecke character~$\psi$ of conductor $\gm=5\sqrt{\delta}\cO=\sqrt{\delta}\cO\cdot \mathfrak{p}_5\cdot \overline{\mathfrak{p}}_5$, where $\mathfrak{p}_5$ is a prime ideal over $5$.
 First we consider the Dirichlet character $\eta\colon (\cO/\mathfrak{m})^*\longrightarrow \C^*$ of  the form $\eta=\eta_1\times\eta_2\times\eta_3$, where $\eta_1$ is the quadratic character of $(\cO/\sqrt{\delta}\cO)^*$, $\eta_2$       is a character of $(\cO/\mathfrak{p}_5)^*$ of order $4$, and $\eta_3$       is the character of
 $(\cO/\overline{\mathfrak{p}}_5)^*$ of order $4$ satisfying $\eta_3(\alpha)=\eta_2(\overline{\alpha})^{-1}$. 
 For $\mathfrak{a}=\mathfrak{p}_{11}^n (\alpha)$, we set
 $\psi(\mathfrak{a})=
 \sqrt[117]{\gamma} \,\,\cdot \eta(\alpha) $.
 One has: (i) 
  $g=\sum_{n>0} b_n q^n =
  \sum_{\mathfrak{a}} \psi(\mathfrak{a}) q^{N(\mathfrak{a})}
  \in S_2(\Gamma_0(25\cdot 31159^2))^{\operatorname{new}}$,
  (ii) the modular abelian variety attached to $g$ has dimension $\dim A_g=234$, and (iii) there is a prime $\gl$ of $E_g$ over $137$ of degree 1 such that $a_p\equiv b_p\pmod{\mathfrak{l}}$  for all primes $p\nmid 2\cdot 5\cdot 31159$.
  
\subsubsection{Some other examples} The next examples are based on the work in \cite{GL} along with Proposition 3.1 in \cite{pep2}.
The tables present data on several examples of $\Q$-curves related to Conjecture~\ref{conj}. These examples provide supporting evidence, but do not amount to a verification of the conjecture.
The notations are as follows. Again
$K=\Q(\sqrt{\delta})$ and denote by $h=h(K)$ its class number.
Let 
$\mathcal{E}\colon Y^2 = X^3+A\,X+B$
be a completely defined $\Q$-curve over $K$ of degree $m$. The conductor of $\mathcal{E}$ is denoted by $\mathcal{N}(\mathcal{E})$, and the
Kodaira types at the bad reduction prime ideals are given in the list 
$\operatorname{Kod}(\mathcal{E})$.
Finally, for every prime $\ell\mid m$, we denote by $\mathfrak{m}=\mathfrak{m}_\ell$ the conductor of 
the Hecke character $\psi_\ell$
corresponding to the
$\ell$-degree isogeny 
character $r=r_\ell$, and $\operatorname{ord}(\eta_\ell)$ is the list of orders at every prime component of the Dirichlet character~$\eta_\ell$ attached to~$\psi_\ell$. We also provide the conductor $\mathfrak{f}(r)$ of the relative cyclic extension  $L=\Qbar^{\ker r}/K$. The level $N$ of the newform associated with the Weil restriction $\operatorname{Res}_{K/\Q}(\mathcal{E})$ satisfies that $N/\Delta_K$ is a generator of
$\mathcal{N}(\mathcal{E})$.
The level $M$ of the newform with CM associated with $\psi_\ell$ is $\operatorname{Norm}_{K/\Q}(\gm_\ell)\Delta_{K/\Q}$.

\subsubsection{Some other examples}

The following examples are based on the work in \cite{GL}, together with Proposition~3.1 in \cite{pep2}. The tables below present data on several examples of $\Q$-curves related to Conjecture~\ref{conj}. These examples provide supporting evidence for the conjecture, but do not constitute a proof.

We fix the following notation. As before, let $K=\Q(\sqrt{\delta})$ and denote by $h=h(K)$ its class number. Let
\[
\mathcal{E}\colon Y^2 = X^3 + A X + B
\]
be a completely defined $\Q$-curve over $K$ of degree $m$. The conductor of $\mathcal{E}$ is denoted by $\mathcal{N}(\mathcal{E})$, and the Kodaira types at the primes of bad reduction are collected in the list $\operatorname{Kod}(\mathcal{E})$.

Finally, for every prime $\ell \mid m$, we denote by $\mathfrak{m}=\mathfrak{m}_\ell$ the conductor of the Hecke character $\psi_\ell$ associated with the $\ell$-degree isogeny character $r=r_\ell$. We also write $\operatorname{ord}(\eta_\ell)$ for the list of local orders of the Dirichlet character $\eta_\ell$ attached to $\psi_\ell$ at the prime components of $\mathfrak{m}$.

We further record the conductor $\mathfrak{f}(r)$ of the relative cyclic extension
$L=\Qbar^{\ker r}/K$. The level $N$ of the newform attached to the Weil restriction $\operatorname{Res}_{K/\Q}(\mathcal{E})$ satisfies that $N/\Delta_K$ generates $\mathcal{N}(\mathcal{E})$.
The level $M$ of the CM newform associated with $\psi_\ell$ is given by
\[
M = \operatorname{Norm}_{K/\Q}(\mathfrak{m}_\ell)\,\Delta_{K/\Q}.
\]

$$
\begin{tblr}{|c|c|c|}
\hline
\delta & h & m \\
\hline
-199 & 9 & 13 \\
\hline
\SetCell[c=3]{l} 
A = 2145/2\,(293-15\,\sqrt{\delta})\,\sqrt{\delta}^2
\\
\hline
\SetCell[c=3]{l}
B =325\,(-230503-11979\,\sqrt{\delta})\,\sqrt{\delta}^3
\\
\hline
\SetCell[c=3]{l} 
\mathcal{N}(\mathcal{E}) =
\mathfrak P_2^2  \mathfrak P_2'^2 
\mathfrak P_3^2
\mathfrak P_5^2  \mathfrak P_5'^2  
\mathfrak P_{13}^2  \mathfrak P_{13}'^2 
\mathfrak P_{\delta}^2
\\
\hline
\SetCell[c=3]{l} 
\operatorname{Kod}(\mathcal{E}) = 
I_{17}^*,I_5^*,III,III,III,II,II,I_{0}^* 
\\
\hline[0.95pt]
\SetCell[c=3]{l}
\mathfrak m = 
\mathfrak P_2^2  \mathfrak P_2'^2 
\mathfrak P_3
\mathfrak P_5  \mathfrak P_5'  
\mathfrak P_{13}  \mathfrak P_{13}' 
\mathfrak P_{\delta}
\\
\hline
\SetCell[c=3]{l}
\operatorname{ord}(\eta) = 
(2,2,4,4,4,6,6,2) 
\\
\hline
\SetCell[c=3]{l}
\mathfrak{f}(r)= 
\mathfrak{m}
\\
\hline
\end{tblr}
%$$
%----------------------------------------
\qquad 
%$$
\begin{tblr}{|c|c|c|}
\hline
\delta & h & m \\
\hline
-431 & 21 & 11 \\
\hline
\SetCell[c=3]{l} 
A = 2586\,(153433-20015\,\sqrt{\delta})
\\
\hline
\SetCell[c=3]{l}
B =24136\,(598268083-7506725\,\sqrt{\delta})
\\
\hline
\SetCell[c=3]{l} 
\mathcal{N}(\mathcal{E}) =
\mathfrak P_2^6  \mathfrak P_2'^6 
\mathfrak P_3^2  \mathfrak P_3'^2  
\mathfrak P_{\delta}^2
\\
\hline
\SetCell[c=3]{l} 
\operatorname{Kod}(\mathcal{E}) = 
I_{8}^*,I_8^*,I_{0}^*,I_{0}^*,I_{0}^*
\\
\hline[0.95pt]
\SetCell[c=3]{l}
\mathfrak m = 
\mathfrak P_2^3  \mathfrak P_2'^3 
\mathfrak P_3  \mathfrak P_3'  
\mathfrak P_{\delta}
\\
\hline
\SetCell[c=3]{l}
\operatorname{ord}(\eta) = 
(2,2,2,2,2) 
\\
\hline
\SetCell[c=3]{l}
\mathfrak{f}(r)= 
\mathfrak{m}\, 
\mathfrak P_{11}
\\
\hline
\end{tblr}
$$
%----------------------------------------
$$
\begin{tblr}{|c|c|c|}
\hline
\delta & h & m \\
\hline
-251 & 7 & 3\cdot 7 \\
\hline
\SetCell[c=3]{l} 
A = 1095/2 -285/2\,\sqrt{\delta}
\\
\hline
\SetCell[c=3]{l}
B =2222-1232\,\sqrt{\delta}
\\
\hline
\SetCell[c=3]{l} 
\mathcal{N}(\mathcal{E}) =
\mathfrak P_2^2 
\mathfrak P_3  \mathfrak P_3' 
\mathfrak P_{\delta}^2
\\
\hline
\SetCell[c=3]{l} 
\operatorname{Kod}(\mathcal{E}) = 
IV^*,I_1,I_{21},I_{0}^* 
\\
\hline[0.95pt]
\SetCell[c=3]{l}
\mathfrak m_3 = 
\mathfrak P_3  \mathfrak P_3' 
\mathfrak P_{\delta}
\\
\hline
\SetCell[c=3]{l}
\operatorname{ord}(\eta_3) = 
(2,2,2) 
\\
\hline
\SetCell[c=3]{l}
\mathfrak{f}(r_3)= 
\gm_3\mathfrak P_3'^{-1}  
\\
\hline[0.95pt]
\SetCell[c=3]{l}
\mathfrak m_7 = 
\mathfrak P_2  \mathfrak P_\delta
\\
\hline
\SetCell[c=3]{l}
\operatorname{ord}(\eta_7) = 
(3,2) 
\\
\hline
\SetCell[c=3]{l}
\mathfrak{f}(r_7)= 
\gm_7
\mathfrak{P}_7
\\
\hline
\end{tblr}
%$$
%----------------------------------------
\qquad 
%$$
\begin{tblr}{|c|c|c|}
\hline
\delta & h & m \\
\hline
-71 & 7 & 2\cdot 5 \\
\hline
\SetCell[c=3]{l} 
A = 1065/2\,(684949 - 29055 \,\sqrt{\delta})
\\
\hline
\SetCell[c=3]{l}
B =1775\,(813956993 + 170710461 \,\sqrt{\delta})
\\
\hline
\SetCell[c=3]{l} 
\mathcal{N}(\mathcal{E}) =
\mathfrak P_2^6
\mathfrak P_2'^6
\mathfrak P_3^2 \mathfrak P_3'^2 
\mathfrak P_5^2 \mathfrak P_5'^2 
\mathfrak P_{\delta}^2
\\
\hline
\SetCell[c=3]{l} 
\operatorname{Kod}(\mathcal{E}) = 
I_2,I_5,I_5^*,I_2^*,III,III,I_0^*
\\
\hline[0.95pt]
\SetCell[c=3]{l}
\mathfrak m_2 = 
\mathfrak P_\delta
\\
\hline
\SetCell[c=3]{l}
\operatorname{ord}(\eta_2) = 
(2) 
\\
\hline
\SetCell[c=3]{l}
\mathfrak{f}(r_2)= (1)=\gm_2 \ \mathfrak P_\delta^{-1}
\\
\hline[0.95pt]
\SetCell[c=3]{l}
\mathfrak m_5 = 
\mathfrak P_3 \mathfrak P_3' \mathfrak P_\delta
\\
\hline
\SetCell[c=3]{l}
\operatorname{ord}(\eta_5) = 
(2,2,2) 
\\
\hline
\SetCell[c=3]{l}
\mathfrak{f}(r_5)= 
\mathfrak m_5
\mathfrak P_5
\\
\hline
\end{tblr}
$$
%----------------------------------------
%\newpage 

\subsection{Elliptic curves over $\Q$ and the modular maximal-cyclic group of order 16}

We consider the elliptic curve $E$ corresponding to the value $t=-4/5$ in the function field of the modular curve $X_G(5)$, as described in Section~\ref{ECmod5}. More precisely, we take
\[
E\colon y^2 = x^3 - 1487920001000\,x + 789770894663059890,
\]
which has conductor
\[
N = 2^6 \cdot 5^2 \cdot 31 \cdot 41 \cdot 223^2 \cdot 251.
\]
The image of the associated mod~$5$ Galois representation is isomorphic to the group $G=[16,6]$, and Serre’s weight is equal to $10$.

Let $f(q)=\sum_{n>0} a_n q^n \in S_2(\Gamma_0(N))^{\operatorname{new}}$ be the normalized newform attached to $E$. The newform $f$ is of CM type mod~$5$ by two imaginary quadratic fields, namely
\[
K=\Q(\sqrt{-1115}) \quad \text{and} \quad K=\Q(\sqrt{-223}).
\]
We denote by $\mathcal{O}$ the ring of integers of either field.

\vskip 0.2cm

Case $K=\Q(\sqrt{-1115})$.
The class group of $K$ is cyclic of order $10$. There exists a normalized newform
\[
g(q)=\sum_{n>0} b_n q^n \in S_2(\Gamma_0(M))^{\operatorname{new}}
\]
with $M=(8\cdot 5\cdot 223)^2$, having CM by $K$, such that $\dim A_g=10$. Moreover, there exists a prime ideal $\mathfrak{l}$ of $\Q(\{b_n\})$ above $5$ with residue degree $1$ such that
\[
a_p \equiv b_p \pmod{\mathfrak{l}}
\quad \text{for all primes } p \neq 31,41,251.
\]
More precisely, the conductor of the associated Hecke character is
\[
\mathfrak{m}=8\sqrt{-1115}\,\mathcal{O},
\]
and its Dirichlet character $\eta : (\mathcal{O}/\mathfrak{m})^* \to \C^*$ decomposes as
\[
\eta = \eta_1 \times \eta_2,
\]
where $\eta_1 : (\mathcal{O}/\sqrt{-1115}\,\mathcal{O})^* \to \C^*$ is the quadratic character, and
$\eta_2 : (\mathcal{O}/8\mathcal{O})^* \to \C^*$ is defined by
\[
\eta_2(\alpha)=\chi(\mathrm{Norm}(\alpha)),
\]
where $\chi$ is the quadratic Dirichlet character of conductor $8$, characterized by $\chi(\pm1)=1$ and~$\chi(\pm3)=-1$.

\vskip 0.2cm

Case $K=\Q(\sqrt{-223})$.
The class group of $K$ is cyclic of order $7$. There exists a normalized newform
\[
g(q)=\sum_{n>0} b_n q^n \in S_2(\Gamma_0(M))^{\operatorname{new}}
\]
with the same level $M=(8\cdot 5\cdot 223)^2$, having CM by $K$, such that $\dim A_g=14$. Moreover, there exists a prime ideal $\mathfrak{l}$ of $\Q(\{b_n\})$ above $5$ of residue degree $1$ satisfying
\[
a_p \equiv b_p \pmod{\mathfrak{l}}
\quad \text{for all primes } p \neq 31,41,251.
\]
More precisely, the conductor of the Hecke character is
\[
\mathfrak{m}=40\sqrt{-223}\,\mathcal{O},
\]
and the associated character decomposes as
\[
\eta=\eta_1\times\eta_2\times\eta_3,
\]
where $\eta_1 : (\mathcal{O}/\sqrt{-223}\,\mathcal{O})^* \to \C^*$ is the quadratic character,
$\eta_2 : (\mathcal{O}/8\mathcal{O})^* \to \C^*$ is defined by $\eta_2(\alpha)=\chi(\mathrm{Norm}(\alpha))$ with $\chi$ as above, and
$\eta_3 : (\mathcal{O}/5\mathcal{O})^* \to \C^*$ is a character of order $3$.

\bibliography{main}{}
\bibliographystyle{alpha}

\end{document}